\documentclass[reqno,10pt]{amsart}
\usepackage{amssymb,amsmath,amsthm,amsfonts,color}
\usepackage{subfigure}

\usepackage{mathrsfs}

\usepackage{ulem}
\usepackage{enumerate}
\usepackage{threeparttable}

\usepackage{bm}
\usepackage[left=2.8cm,right=2.8cm,top=2.8cm,bottom=2.8cm]{geometry}
\usepackage{mathtools}
\usepackage{graphicx}
\usepackage[format=hang,labelfont=bf]{caption}

\usepackage[colorlinks=true, pdfstartview=FitV, linkcolor=blue,
            citecolor=blue, urlcolor=blue]{hyperref}

\allowdisplaybreaks
\numberwithin{equation}{section}
\theoremstyle{plain}
\newtheorem{theorem}{Theorem}[section]
\newtheorem{proposition}[theorem]{Proposition}
\newtheorem{lemma}[theorem]{Lemma}
\newtheorem{example}[theorem]{Example}
\newtheorem{corollary}[theorem]{Corollary}
\theoremstyle{definition}

\newtheorem{remark}[theorem]{Remark}

\newtheorem*{theorem*}{Theorem}

\newcommand{\calk}{\mathcal K}
\newcommand{\frkm}{\mathfrak m}

\begin{document}
\title{Stability analysis of an extended quadrature method of moments for kinetic equations}

\author{Ruixi Zhang}
\address{Department of Energy and Power Engineering, Tsinghua University\\
    Beijing, 100084, China}
\email{1553548358@qq.com}

\author{Qian Huang*}
\address{Department of Energy and Power Engineering, Tsinghua University\\
    Beijing, 100084, China}
\thanks{* Corresponding author}
\email{huangqian@tsinghua.edu.cn; hqqh91@qq.com}

\author{Wen-An Yong}
\address{Department of Mathematical Sciences, Tsinghua University\\
    Beijing, 100084, China \\
    Yanqi Lake Beijing Institute of Mathematical Sciences and Applications\\
    Beijing 101408, China}
\email{wayong@tsinghua.edu.cn}


\vskip .2truecm
\begin{abstract}
    This paper performs a stability analysis of a class of moment closure systems derived with an extended quadrature method of moments (EQMOM) for the one-dimensional BGK equation. The class is characterized with a kernel function. A sufficient condition on the kernel is identified for the EQMOM-derived moment systems to be strictly hyperbolic. We also investigate the realizability of the moment method. Moreover, sufficient and necessary conditions are established for the two-node systems to be well-defined and strictly hyperbolic, and to preserve the dissipation property of the kinetic equation.
\end{abstract}

\maketitle

\normalem

\section{Introduction}
\label{sec:intro}

By describing the evolution of problem-specific distribution functions, the kinetic models are founded with a solid basis in a wide range of complex interacting systems.
For instance, the celebrated Boltzmann equation governs the distribution of the molecular velocity and is believed to better characterize the rarefied flows where their hydrodynamic counterparts, being Euler and/or Navier-Stokes equations, become less reliable \cite{Harris2004}.
However, solving the kinetic equation is challenging.
For one thing, the binary collision operator of the Boltzmann equation causes quadratic costs while treating the velocity dependence.
A widely-accepted measure is to apply the BGK operator which models the collision as a relaxation process towards the local equilibrium \cite{bgk1954}.
This model not only reduces the computational costs, but also has the desired conservation laws and an $H$-theorem characterizing the dissipation properties due to collisions \cite{Harris2004}.

On the other hand, the high dimension of the phase space raises significant difficulties in computation, even for the BGK equation.
Several deterministic methods have thus been developed \cite{Dima2014,Mieu2000}, including the discrete velocity model, the spectral methods, and various methods of moments, to remove the dependence upon the molecular velocity and deduce spatial-time models of macroscopic variables from the kinetic equation.
This paper focuses on the method of moments. In the moment system, only the lower-order moments have clear physical interpretations after being related to the density, mean velocity and temperature (internal energy) of the system \cite{Harris2004}.
The higher-order moments may provide additional information beyond the classical hydrodynamic models.

All moment systems need to be closed, which is mostly done by reconstructing the distribution from the transported moments \cite{MarFox2013}.
The well-known Grad’s 13-moment theory was established based on a linear expansion of the Maxwellian \cite{Grad1949}, but far away from equilibrium, this ansatz can lead to negative values of distributions.
In contrast, the quadrature-based method of moments \cite{Mc1997} was proposed with nonlinear reconstruction of the distribution function, and such a treatment seems to be suitable for non-equilibrium flows.
However, approximating the distribution with a linear combination of multiple Dirac $\delta$-functions with unknown centers \cite{Mc1997}, the method (called QMOM) fails in the simulation of BGK equation due to the occurrence of singularity \cite{Fox2008}.

This highlights the importance of understanding the mathematical properties of the moment closure systems, which are usually first-order PDEs. For real-world physical models, the PDE is expected to be hyperbolic so that the system is robust against small perturbations of the initial data \cite{Serre1999}.
Indeed, the unphysical behaviors of the Grad's theory and the QMOM approach can both be attributed to the lack of hyperbolicity \cite{Cai13,Chalons2012}.
Let us mention some efforts to achieve hyperbolic regulation of the Grad's theory \cite{Cai13,Struchtrup2003} and other moment closure systems \cite{KoRo2020}.
Furthermore, a thorough stability analysis should as well account for the source term of the model, as the Boltzmann and BGK collisions are both featured with $H$-theorems \cite{Harris2004}.
For the moment closure system, it is believed that the structural stability condition proposed in \cite{Yong1999} for hyperbolic relaxation systems is a proper characterization of the dissipation property.
The condition specifies how the source term should be coupled with the hyperbolic part in the vicinity of the equilibrium.
Admitting such a structure, the resultant moment system is compatible with the classical theories \cite{Yong1999}.
Recently, it was shown that the structural stability condition is fulfilled by many moment closure systems, including the hyperbolic regularization models of rarefied gases \cite{Di2017,Zhao2017} and a series of hyperbolic shallow water moment models \cite{Huang2022}.

The objective of this paper is to investigate the stability properties of the extended QMOM (termed EQMOM) for the BGK equation. In EQMOM, the velocity distribution is reconstructed as a sum of multiple continuous kernel functions instead of the $\delta$-function \cite{Chalons2010,Yuan2012}.
A list of kernels that can be used for the purpose of EQMOM is summarized in \cite{Pigou2018}.
If the kernel is the Gaussian distribution (denoted Gaussian-EQMOM), it was found that the method is well-defined (namely, the unclosed terms can be uniquely determined), and the resultant moment system respects the structural stability condition \cite{Huang2020}.
However, little is known for other types of kernels, which may be suitable for different scenarios.
For instance, the space plasmas follow kappa distributions with high energy tails deviated from a Maxwellian \cite{PL2010}.
We also remark that the EQMOM with the beta function as the kernel is used to treat a simplified radiative transfer equation \cite{ALL2016}.

This paper deals with the EQMOM induced by a univariate kernel function (see (\ref{eq:feqmom})) and considers the spatial one-dimensional (1-D) BGK equation.
Surprisingly, it is found that not all kernels share the nice properties of the Gaussian kernel. As the main result of this paper, we reveal the exact constraints on the kernel such that the two-node EQMOM ($n=2$) is well-defined, that the resultant moment system is strictly hyperbolic, and that the system properly respects the dissipation property (as judged by the structural stability condition).
These constraints are inequalities of the moments of the kernel function, and are thus easy to check for specific kernels.
Moreover, a sufficient condition is estalished for the $n$-node EQMOM-derived moment system to be strictly hyperbolic.
It is worth mentioning that, in our argument, the realizable moment set can be characterized for general $n\ge 1$, which essentially contains the results in \cite{Chalons2017} (Proposition 3.1 for the two-node Gaussian-EQMOM) as a special case.
This leads to a polynomial which is an answer to the question raised in \cite{MarFox2013} (footnote 16 on p.88).
Furthermore, we present abundant examples of the kernel functions to which our theory can be applied in a straightforward manner.

The remainder of the paper is organized as follows. Our main results are presented in Section \ref{sec:pre}.
Section \ref{sec:inj} is devoted to studying the realizability of the extended quadrature method of moments.
Hyperbolicity of the resultant moment systems is analyzed in Section \ref{sec:hyp}.
For two-node systems, the structural stability condition is verified in Section \ref{sec:stab}.
Section \ref{sec:exp} presents a number of specific kernel functions and investigates their behaviors numerically.
The conclusions are given in Section \ref{sec:conclusions}.

\section{Preliminaries and main results}
\label{sec:pre}

Consider the hypothetical 1-D BGK equation for the distribution $f=f(t,x,\xi)$ with time $t \in \mathbb R_+$, position $x \in \mathbb R$ and velocity $\xi \in \mathbb R$:
\begin{equation} \label{eq:1D-BGK}
  \left \{
  \begin{aligned}
    \partial_t f+\xi \partial_x f &= \frac{1}{\tau}(f^{eq}-f), \\
    f^{eq} = f^{eq}(\rho,U,\theta;\xi) &= \frac{\rho}{\sqrt{2\pi\theta}}\exp \left(-\frac{(\xi-U)^2}{2\theta} \right).
  \end{aligned}
  \right.
\end{equation}
Here $\tau$ is a relaxation time. In the local Maxwellian equilibrium $f^{eq}$, the density $\rho$, mean velocity $U$ and temperature $\theta$ are determined as the velocity moments of $f$:
\[
  \rho = \int_{\mathbb R} fd\xi,\quad
  \rho U = \int_{\mathbb R} \xi fd\xi,\quad
  \rho\theta + \rho U^2 = \int_{\mathbb R} \xi^2fd\xi.
\]

Define the $j$th velocity moments of $f$ as
\[ M_j = M_j(t,x) = \int_{\mathbb R} \xi^jfd\xi \]
for $j\in\mathbb{N}$.
The evolution equation for $M_j$ can be immediately derived from (\ref{eq:1D-BGK}) as
\begin{equation} \label{eq:unc}
\partial_tM_j+\partial_xM_{j+1}= \frac{1}{\tau} \left(\rho \Delta_j^{eq}(U,\sqrt{\theta})-M_j \right)
\end{equation}
with
\[
  \Delta_j^{eq}(u,\sigma) = \int_{\mathbb R} \xi^j \frac{1}{\sqrt{2\pi}\sigma} \exp \left( -\frac{(\xi-u)^2}{2\sigma^2} \right) d\xi.
\]
Notice that the first $n$ equations for $M_0,...,M_{n-1}$ contain the term $\partial_x M_n$. Therefore, any finite truncation of the above equations leads to an unclosed system, and a closure procedure is required.
In this paper, we are concerned with an extended quadrature method of moments (EQMOM) \cite{Chalons2010,Yuan2012}.

\subsection{Extended quadrature method of moments (EQMOM)}

Let $\calk = \calk(\xi) \ge 0$ satisfy
\begin{equation} \label{eq:ker_cond}
  \frkm_j:=\int_{\mathbb R}\xi^j \mathcal K(\xi)d\xi < \infty, \ \forall j \in \mathbb N, \quad \text{and } \frkm_0 = 1.
\end{equation}
In EQMOM, the distribution $f$ is approximated with the following ansatz
\begin{equation} \label{eq:feqmom}
  f(\xi) = \sum_{i=1}^n \frac{w_i}{\sigma} \calk \left(\frac{\xi-u_i}{\sigma}\right)
\end{equation}
with the weights $w_i$, nodes $u_i$ and `width' $\sigma>0$ to be determined. To do this, the first ($2n+1$) lower-order moments are employed:
\begin{equation} \label{eq:MWmap}
  M_j=\sum_{i=1}^nw_i\Delta_j(u_i,\sigma), \quad\textrm{with }
  \Delta_j(u,\sigma) := \int_{\mathbb R} \xi^j \frac{1}{\sigma} \calk \left(\frac{\xi-u}{\sigma} \right) d\xi
\end{equation}
for $j=0,\dots,2n$.
This defines a map $M:= (M_0,...,M_{2n})^T = \mathcal M (W)$ for $W=(w_1,u_1,...,w_n,u_n,\sigma)^T \in \mathbb R^{2n+1}$ with $\sigma>0$.
Here the superscript `$T$' denotes the transpose of a vector or matrix.

Suppose $\mathcal M$ is injective on a certain domain $W \in \Omega \subset \mathbb R^{2n+1}$. Then for any $M \in\mathbb{G}:= \mathcal M (\Omega)$, there exists a unique $W = \mathcal M^{-1}(M)$ solving (\ref{eq:MWmap}).
In this way, the EQMOM is \textit{well-defined} and the next moment $M_{2n+1}$ can be evaluated as a function of $M$:
\begin{equation} \label{eq:mlast}
   M_{2n+1}=\mathcal{M}_{2n+1}(M) :=\sum_{i=1}^nw_i\Delta_{2n+1}(u_i,\sigma).
\end{equation}
Consequently, the first ($2n+1$) equations in (\ref{eq:unc}) are closed as a system of first-order PDEs:
\begin{equation} \label{eq:eqmomsys}
  \partial_t M + \partial_x (M_1,...,M_{2n+1})^T = \frac{1}{\tau} \left(\rho\Delta^{eq}(U,\sqrt{\theta}) -M \right).
\end{equation}
Here $\Delta^{eq}(U,\sqrt{\theta})=\left(\Delta_0^{eq}(U,\sqrt{\theta}),\dots,\Delta_{2n}^{eq}(U,\sqrt{\theta})\right)^T \in \mathbb R^{2n+1}$ with $\rho=M_0$, $\rho U=M_1$ and $\rho(\theta+U^2)=M_2$.

The main goal of this paper is to investigate the injectivity of the map $\mathcal M$ in (\ref{eq:MWmap}) for the general kernel $\calk(\xi)$, where the injectivity is closely related to the realizability of moments.
Such analyses are useful to design efficient and robust algorithms to solve $W$ from (\ref{eq:MWmap}).
Moreover, we analyze hyperbolicity of the moment closure system (\ref{eq:eqmomsys}) and its dissipation property inherited from the the $H$-theorem of the kinetic equation.

\begin{remark} \label{rem:norm}
  For the sake of simplicity, throughout this paper we  assume that the kernel function $\calk(\xi)$ is \textit{normalized}, in the sense that $\frkm_0=1$, $\frkm_1=0$ and $\frkm_2=1$.
  In fact, a general $\calk(\xi)$ can be normalized as $\calk^+ (\xi) = h \calk (h\xi + \xi_0)$ with $\xi_0=\frkm_1$ and $h = \sqrt{\frkm_2-\frkm_1^2}$ (note that $\frkm_2 \ge \frkm_1^2$ due to the Cauchy-Schwartz inequality).
  Moreover, the map $\mathcal M_{2n+1}(M)$ in (\ref{eq:mlast}) is the same as that derived from $\calk(\xi)$, so the moment closure system (\ref{eq:eqmomsys}) is unchanged with the normalization.
\end{remark}

\subsection{Structural stability condition} \label{sec:ssc}

For smooth solutions, the balance laws (\ref{eq:eqmomsys}) can be written as
\begin{equation} \label{eq:1DPDE}
  \partial_t M+A(M)\partial_x M = S(M) := \frac{1}{\tau} \left(\rho\Delta^{eq}(U,\sqrt{\theta}) -M \right)
\end{equation}
with coefficient matrix
\begin{equation} \label{eq:defA}
A(M)=
\begin{bmatrix}
  0 & 1 &&&\\
  & 0 & 1&&\\
  && \ddots & \ddots &\\
  &&& 0 & 1\\
  a_0 & a_1 & \cdots & a_{2n-1} &a_{2n}
\end{bmatrix},
\end{equation}
where $a_j(M)= \partial \mathcal M_{2n+1} / \partial M_j$ for $j=0,\dots,2n$.
It is called \textit{hyperbolic} if $A(M) \in \mathbb R^{(2n+1) \times (2n+1)}$ has $(2n+1)$ linearly-independent real eigenvectors \cite{Serre1999}. If $A(M)$ has $(2n+1)$ distinct real eigenvalues, it is called \textit{strictly hyperbolic}. Obviously, strict hyperbolicity implies hyperbolicity.
The dissipativeness of the moment system can be characterized with the structural stability condition proposed in \cite{Yong1999} for hyperbolic relaxation systems.

Assume that the equilibrium manifold $\mathcal E = \{M \in \mathbb G \mid S(M)=0\}$ is not empty. Denote by $S_M(M)$ the Jacobian matrix of $S(M)$. The structural stability condition reads as
\begin{itemize}
    \item [(I)] For any $M \in \mathcal E$, there exist invertible matrices $P=P(M)\in \mathbb R^{(2n+1) \times (2n+1)}$ and $\hat T = \hat T(M) \in \mathbb R^{r \times r}$ ($0<r \le 2n+1$) such that
    \[
    P S_M(M) P^{-1} = \text{diag}(\textbf{0}_{(2n+1-r) \times (2n+1-r)}, \hat T).
    \]

    \item [(II)] For any $M \in \mathbb G$, there exists a positive definite symmetric matrix $A_0 = A_0(M)$ such that $A_0 A(M) = A^T(M) A_0$.

    \item [(III)] For any $M \in \mathcal E$, the coefficient matrix and the source are coupled as
    \[
        A_0 S_M (M) + S_M^T (M) A_0 \le - P^T
        \begin{bmatrix}
            0 & 0 \\
            0 & I_r
        \end{bmatrix}
        P.
    \]
\end{itemize}
Here $I_r$ is the unit matrix of order $r$.

\begin{remark} \label{rem:sym}
  Recently, it has been demonstrated that several moment models from the kinetic equations respect the structural stability condition, including the Gaussian-EQMOM \cite{Huang2020} and the hyperbolic regularization models \cite{Di2017,Zhao2017}.
  For the 1-D system (\ref{eq:1DPDE}), Condition (II) is satisfied if and only if the system is hyperbolic \cite{Huang2020}.
  Condition (III) can be regarded as a proper manifestation of the dissipation property inherited from the kinetic model. See detailed discussions in \cite{Yong2001}.
\end{remark}

\subsection{Main results}

Our main results are collected in this subsection.
As mentioned in Remark \ref{rem:norm}, we assume that the kernel $\calk(\xi)$ is normalized, that is, $\frkm_0=\frkm_2=1$ and $\frkm_1=0$.

To state the results, we recursively define a sequence of numbers associated with $\{\frkm_j\}:$
\begin{equation} \label{eq:bj}
  b_0=1,\quad b_j= - \sum_{k=1}^j \frac{\frkm_k}{k!}b_{j-k} \quad \text{for } j=1,2,\dots
\end{equation}
and a number of auxiliary moments associated with  $M=(M_0,...,M_{2n})^T \in \mathbb R^{2n+1}$:
\begin{equation} \label{eq:Mstar}
 M^*=\{M_j^*\}_{j=0}^{2n},\quad
 M_j^*=M_j^*(\sigma) = \sum_{k=0}^j b_k \sigma^k \frac{j!}{(j-k)!} M_{j-k}.
\end{equation}
Moreover, for $M=(M_0,...,M_{2n})^T$ we introduce the Hankel matrix \cite{gtm277} as
\begin{equation}
H_k(M)=
\begin{bmatrix}
  M_0 & M_1 & \cdots & M_k \\
  M_1 & M_2 & \cdots & M_{k+1} \\
  \vdots & \vdots & \ddots & \vdots \\
  M_k & M_{k+1} & \cdots & M_{2k}
\end{bmatrix} \in \mathbb R^{(k+1)\times (k+1)} \quad
\text{for } k \le n.
\end{equation}
This is a real symmetric matrix.

Our first result is
\begin{theorem} \label{prop:sigpoly}
  Given $M=(M_0,\dots,M_{2n})^T \in \mathbb R^{2n+1}$, set
  \begin{equation} \label{eq:defsigpoly}
   P_n(\sigma;M) = \det H_n(M^*(\sigma)).
  \end{equation}
  The following statements are equivalent.

  (i). There exists a unique $W \in \Omega = \Omega'\times\{\sigma>0\}$ with
  \begin{equation} \label{eq:Omega}
   \Omega'=\{(w_1,u_1,...,w_n,u_n)\in\mathbb{R}^{2n}\mid \ w_i>0, \ \forall i; \ u_1< \dots < u_n \},
  \end{equation}
  such that $\mathcal M(W)=M$.

  (ii). $P_n(\sigma;M)=0$ has a unique positive root $\sigma_0$ such that the Hankel matrix $H_{n-1}(M^*(\sigma_0))$ is positive definite.
\end{theorem}

\begin{remark}
  Clearly, a similar conclusion can be formulated when some of the weights $w_i$ are zero or the centers $u_i$ coincide.
  In that case, one can find a unique index $k\le n$ such that $w_i>0$ for any $i=1,\dots,k$ and all the $u_i$'s are distinct.
\end{remark}

\begin{remark} \label{rem:invM}
  Statement (ii) serves as an implicit realizable condition for $M$.
  For $n=2$, this condition is exactly that in \cite{Chalons2017} for the Gaussian kernel.
  Moreover, Theorem \ref{prop:sigpoly} suggests a key step in inverting the map $M=\mathcal M(W)$, namely, finding $\sigma$ as a root of the polynomial $P_n(\sigma) = \det H_{n}(M^*(\sigma))$ of degree $n(n+1)$.
  For $n=2$, the polynomial for even and normalized kernels is
  \[
    P_2(\sigma)= (5-\frkm_4)\sigma_1^3 + 2\theta(3-\frkm_4)\sigma_1^2 + (M_4'-\frkm_4\theta^2)\sigma_1 + M_3'^2
  \]
  with $\sigma_1=\sigma^2-\theta$,
  \[
    M_3' = \frac{M_3}{M_0} - 3U\theta - U^3, \quad
    M_4' = \frac{M_4}{M_0} - 4UM_3' - 6U^2\theta - U^4,
  \]
  $U=M_1/M_0$ and $\theta = M_2/M_0 - U^2$.
  Once $\sigma$ is found, other components of $W$ can be determined by the existing algorithms \cite{MarFox2013} (see Remark \ref{rem:3.1}).
\end{remark}

As a corollary of this theorem, we have
\begin{corollary}[Injectivity] \label{thm:inj}
  For $n=2$, the map $M=\mathcal M(W)$ in (\ref{eq:MWmap}) is injective for $W \in \Omega$ if and only if the inequality $\frkm_4 \ge 3 + \frac{9}{8} \frkm_3^2$ holds for $\calk(\xi)$.
\end{corollary}

\begin{remark}\label{rem:Mtot}
  For $n=2$, it is not difficult to see that $M_5$ can be expressed in terms of $M_0,...,M_4$ when $u_1=u_2$. As a consequence, the condition of Corollary \ref{thm:inj} ensures that the EQMOM is well defined on $\mathcal M(\Omega^{tot})$ with
  \[
    \Omega^{tot}=\{W \in \mathbb R^5 | w_1>0, \ w_2>0,\ u_1\leq u_2,\ \sigma>0\}.
  \]
\end{remark}

Suppose the map $M=\mathcal M(W)$ is injective on $\Omega$ for general $n$. Our second result is
\begin{theorem}[Hyperbolicity] \label{prop:schy2}
  If the $b$-polynomial
  \[
    p(t):=\sum_{j=0}^{2n+1}b_{2n+1-j}t^j
  \]
  has $(2n+1)$ real roots (counting multiplicity) and at least two roots are nonzero, then the $n$-node EQMOM moment system (\ref{eq:eqmomsys}) is strictly hyperbolic on $\mathcal M(\Omega)$.
\end{theorem}

Furthermore, for even kernels we have
\begin{theorem}[Hyperbolicity] \label{thm:hyp2}
  Let the kernel $\calk(\xi)$ be an even function. The two-node EQMOM moment system (\ref{eq:eqmomsys}) is strictly hyperbolic for $M\in \mathcal M(\Omega^{tot})$ if and only if $3\le \frkm_4 < 6$.
\end{theorem}

\begin{theorem}[Dissipativeness] \label{thm:stab}
  Let the kernel $\calk(\xi)$ be an even function and $3\le \frkm_4 < 6$. Then the two-node EQMOM moment system (\ref{eq:eqmomsys}) satisfies the structural stability condition if and only if $3\le \frkm_4<5$.
\end{theorem}

In the next section, Theorem \ref{prop:sigpoly} and Corollary \ref{thm:inj} will be proved. Section \ref{sec:nhyp} is devoted to a proof of Theorem \ref{prop:schy2}, while Theorems \ref{thm:hyp2} \& \ref{thm:stab} are proved in Sections \ref{sec:hyp2} and \ref{sec:stab}, respectively.

\section{Injectivity} \label{sec:inj}

In this section, we prove Theorem \ref{prop:sigpoly} and Corollary \ref{thm:inj}.
To start with, we recall the definition of the map $M =\mathcal M(W)$ in (\ref{eq:MWmap}):
\[
  M_j=\sum_{i=1}^nw_i\Delta_j(u_i,\sigma), \quad\textrm{with }
  \Delta_j(u,\sigma) := \int_{\mathbb R} \xi^j \frac{1}{\sigma} \calk \left(\frac{\xi-u}{\sigma} \right) d\xi,
\]
for $j=0,\dots,2n$.

By performing the change of variables $\frac{\xi-u}{\sigma} \mapsto \xi$, we can easily see that
\begin{equation} \label{eq:deltajex}
  \Delta_j(u,\sigma) = \sum_{k=0}^j \binom{j}{k} \frkm_k \sigma^k u^{j-k},
\end{equation}
indicating that $\Delta_j(u, \sigma)$ is a homogeneous bivariate polynomial of $u$ and $\sigma$.
It is not difficult to verify
\begin{subequations}
\begin{align}
  \partial_u \Delta_j(u,\sigma) &= j \Delta_{j-1}(u,\sigma), \label{eq:dudelta} \\
  \partial_{\sigma} \Delta_j(u,\sigma) &= j\sum_{k=0}^{j-1} \binom{j-1}{k} \frkm_{k+1} \sigma^k u^{j-1-k}. \label{eq:dsigdelta}
\end{align}
\end{subequations}

Notice that $M=(M_0,...,M_{2n})^T$ can be conversely expressed in terms of the auxiliary moments $M_j^*=M_j^*(\sigma)$ defined in (\ref{eq:Mstar}) as
\begin{equation} \label{eq:MfromMstar}
  M_j = \sum_{k=0}^j \binom{j}{k} \frkm_k \sigma^k M_{j-k}^* \quad \text{for } j=0, \dots, 2n.
\end{equation}
Indeed, a straightforward calculation of the right-hand side, incorporating (\ref{eq:Mstar}), yields
\[
  \begin{aligned}
    \text{r.h.s.} &= \sum_{k=0}^j \binom{j}{k} \frkm_k \sigma^k \sum_{l=0}^{j-k} b_l \sigma^l \frac{(j-k)!}{(j-k-l)!} M_{j-k-l} \\
    &= \sum_{s=0}^j \frac{j!}{(j-s)!}\sigma^s M_{j-s}\sum_{k=0}^s \frac{\frkm_k}{k!} b_{s-k}.
  \end{aligned}
\]
The second equality is derived after the change of variables $s=k+l$. Rewriting (\ref{eq:bj}) as
\[
  \sum_{k=0}^s \frac{\frkm_k}{k!} b_{s-k} = 0,\quad s\ge 1,
\]
we obtain $\text{r.h.s}=M_j$.
Similarly, with (\ref{eq:deltajex}) involved, a direct calculation of the right-hand side of (\ref{eq:Mstar}) results in
\begin{equation} \label{eq:defmstar}
    M_j^*(\sigma)=\sum_{i=1}^nw_iu_i^j,\quad j=0,...,2n.
\end{equation}

\begin{remark} \label{rem:3.1}
  Together with Remark \ref{rem:invM}, the last formula suggests a practical method to solve $W\in \mathbb R^{2n+1}$ from the EQMOM map $M = \mathcal M(W)$. Once $\sigma$ has been determined as shown in Remark \ref{rem:invM}, the $(w_i,u_i)$'s can be determined by solving the first $2n$ equations in (3.4) with existing algorithms \cite{MarFox2013}.
\end{remark}

About the Hankel matrix, we quote the following lemma.
\begin{lemma}[\cite{gtm277}, Theorem 9.7] \label{lem:hankpos}
  Given $M'=(M_0,\dots,M_{2n-1})\in \mathbb R^{2n}$, if the nonlinear equations
  \[
    \sum_{i=1}^n w_i u_i^j = M_j \quad \text{for } j=0, \dots, 2n-1
  \]
  have a solution in $\Omega'$ defined in (\ref{eq:Omega}), then the Hankel matrix $H_{n-1}(M')$ is positive definite. Conversely, if $H_{n-1}(M')$ is positive definite, then the last equations have a unique solution in $\Omega'$.
\end{lemma}

To prove Theorem \ref{prop:sigpoly}, we first notice the following fact:
\begin{proposition} \label{prop:hanksing}
  If $M_j^* = \sum_{i=1}^n w_i u_i^j$ for $j=0,\dots,2n$, the Hankel matrix $H_n(\{M_j^*\})$ is singular.
\end{proposition}

\begin{proof}
  A direct calculation gives
  \[
  \det
  \begin{bmatrix}
    M_0^* & \cdots & M_n^* \\
    \vdots & & \vdots \\
    M_n^* & \cdots & M_{2n}^*
  \end{bmatrix}
  = \sum_{1\leq i_0,...,i_n\leq n} \det
  \begin{bmatrix}
    w_{i_0} & \cdots & w_{i_n}u_{i_n}^n \\
    \vdots & & \vdots \\
    w_{i_0}u_{i_0}^n & \cdots & w_{i_n}u_{i_n}^{2n}
  \end{bmatrix}.
  \]
  For each determinant in the summation, at least two of the $(n+1)$ indices $1\le i_0,\dots,i_n\le n$ are identical, therefore the determinant is zero. Hence the Hankel matrix is singular.
\end{proof}

\begin{proof}[Proof of Theorem \ref{prop:sigpoly}]
  (i) $\Rightarrow$ (ii). In this case, it has been shown in (\ref{eq:defmstar}) that $M^*_j(\sigma)$ can be expressed as $\sum_{i=1}^n w_i u_i^j$ for $j=0,\dots,2n$ with $w_i> 0$ and all the $u_i$'s distinct.
  Then it follows from Proposition \ref{prop:hanksing} that $P_n(\sigma;M)=\det H_n(\{M_j^*(\sigma)\}) = 0$.
  Because all $w_i>0$ and the $u_i$'s are distinct, we deduce from Lemma \ref{lem:hankpos} that $H_{n-1}(\{M_j^*(\sigma)\})$ is positive definite.

  For the uniqueness, suppose that $P_n(\sigma;M)$ has another root $\sigma_1>0$ such that $H_{n-1}(\{M_j^*(\sigma_1)\})$ is positive definite.
  It follows from Lemma \ref{lem:hankpos} that there exists a $2n$-tuple $\{q_i>0,v_i\}_{1\leq i\leq n}$ such that $v_1<\cdots<v_n$ and $M_j^*(\sigma_1)=\sum_{i=1}^nq_iv_i^j$ for $j=0,...,2n-1$.
  Set $S_j=\sum_{i=1}^n q_i v_i^j$. From $P_n(\sigma_1;M)=0$ and Proposition \ref{prop:hanksing} we see that $\det H_{n}(\{M_j^*(\sigma_1)\}) = P_n(\sigma_1;M) = \det H_{n}(\{S_j\}) = 0$.
  On the other hand, we observe that
  \[
    \det H_{n}(\{M_j^*(\sigma_1)\}) = a M_{2n}^*(\sigma_1) + b, \quad
    \det H_{n}(\{S_j\}) = a S_{2n} + b
  \]
  with
  \[
    a = \det H_{n-1}(\{M_j^*\}) = \det H_{n-1}(\{S_j\})>0
  \]
  and $b$ depending only on $M^*_j(\sigma_1)=S_j$ with $j\le 2n-1$.
  Thus, we have $M^*_{2n}(\sigma_1)=S_{2n}$ and thereby get another solution to the equations $\mathcal{M}(W')=M$, violating the uniqueness in (i). This proves (ii).

  (ii) $\Rightarrow$ (i). Assume that $P_n(\sigma_0;M)=0$ and $H_{n-1}(\{M_j^*(\sigma_0)\})$ is positive definite. The reasoning above shows that there exists a unique $2n$-tuple $\{w_i>0,u_i\}$ such that $u_1<\cdots< u_n$ and $W=(w_i,u_i,\sigma_0)$ solves $\mathcal{M}(W)=M$.
  If $W_1=(q_i,v_i,\sigma_1)\neq W$ is another solution, then $\sigma_1\neq\sigma_0$ and the reasoning in (i) shows that $\sigma_1$ is another root of $P_n(\sigma;M)=0$ which contradicts (ii). This completes the proof.
\end{proof}

Now we are in a position to prove Corollary \ref{thm:inj}.

\begin{proof}[Proof of Corollary \ref{thm:inj}]
  Assume  $\frkm_4 \ge 3 + \frac{9}{8} \frkm_3^2$ for $\calk(\xi)$. It suffices to show that the Jacobian $\frac{\partial \mathcal M}{\partial W}$ is invertible for $W\in \Omega$.
  Recall the explicit expressions of $\Delta_j(u,\sigma)$ in (\ref{eq:deltajex}) and its derivatives in (\ref{eq:dudelta}) \& (\ref{eq:dsigdelta}).
  Using
  \[
    M_j = w_1\Delta_j(u_1,\sigma)+w_2\Delta_j(u_2,\sigma),\quad j=0,1,\dots,4,
  \]
  we compute the $(j+1)$th row of the Jacobian $\frac{\partial \mathcal M}{\partial W}$ as
  \[
    \left(\Delta_j(u_1,\sigma), jw_1\Delta_{j-1}(u_1,\sigma), \Delta_j(u_2,\sigma), jw_2\Delta_{j-1}(u_2,\sigma), \sum_{i=1}^2 w_i \partial_{\sigma} \Delta_j(u_i,\sigma) \right)
  \]
   and, by resorting to MATLAB,
   \[
    \det \frac{\partial \mathcal M}{\partial W} = w_1 w_2 \sigma^3 (u_1-u_2)^4 \left[ w_1 q\left( \frac{u_1-u_2}{\sigma} \right) + w_2 q\left( \frac{u_2-u_1}{\sigma} \right) \right]
   \]
   with $q(x)= x^2 + 3 \frkm_3 x + 2(\frkm_4 - 3)$.
   Obviously, we have $q(x)\ge 0$ for any real $x$ due to $\frkm_4 \ge 3 + \frac{9}{8} \frkm_3^2$, and $q(x)$ does not have two distinct zeros. Therefore, we have $\det \frac{\partial \mathcal M}{\partial W} >0$ for $W \in \Omega$.

   Conversely, we assume $\frkm_4 < 3 + \frac{9}{8} \frkm_3^2$ for $\calk(\xi)$ and show that $\mathcal M$ is not injective on $\Omega$.
   To do this, we set $M_b=\mathcal{M}(\frac{1}{2},0,\frac{1}{2},0,1)=\{\frkm_j\}_{j=0}^4$ and compute
   \begin{align*}
    &M_0^*=1,\quad M_1^*=0,\quad M_2^*(\sigma)=1-\sigma^2,\\
    &M_3^*(\sigma)=\frkm_3(1-\sigma^3),\quad M_4^*(\sigma)=\frkm_4(1-\sigma^4)-6\sigma^2(1-\sigma^2)
   \end{align*}
   and
   \[
   \begin{aligned}
    \frac{P_2(\sigma;M_b)}{(1-\sigma)^2} =&
    (\frkm_4-\frkm_3^2-5)\sigma^4 + 2(\frkm_4-\frkm_3^2-5)\sigma^3 + (2\frkm_4-3\frkm_3^2-6)\sigma^2 \\
    &+ 2(\frkm_4-\frkm_3^2-1)\sigma+(\frkm_4-\frkm_3^2-1).
   \end{aligned}
   \]
   The corresponding $2\times2$ Hankel matrix, which we rewrite as $H_1(\sigma;M_b)$, is positive definite if and only if the polynomial
   \[
    \det H_1(\sigma;M_b) = M^*_0(\sigma)M^*_2(\sigma)-M^{*2}_1(\sigma)=1-\sigma^2>0,
   \]
   namely, $0<\sigma<1$. On the other hand, we denote $\tilde{P}_2(\sigma;M)=\frac{P_2(\sigma;M)}{(1-\sigma)^2}$ and notice
   \begin{displaymath}
   \tilde{P}_2(1;M_b) = 8\frkm_4 - 9\frkm_3^2 - 24<0\quad\text{and}\quad\tilde{P}_2(0;M_b) = \frkm_4 - \frkm_3^2 -1 = \det H_2(M_b)>0
   \end{displaymath}
   where the second inequality follows from the positivity of $H_2(M_b)$.
   By continuity, we may choose $\epsilon>0$ and $M\in\mathcal{M}(\Omega'\times\{\sigma=1\})$ such that
   \begin{displaymath}
    P_2(\epsilon;M)>0>P_2(1-\epsilon;M)\quad\text{and}\quad\det H_1(\sigma;M)>0
   \end{displaymath}
   for $\sigma\in[\epsilon,1-\epsilon]$. Therefore, $P_2(\sigma;M)$ has a root $\sigma_0$ in $(\epsilon,1-\epsilon)$. Clearly, $\sigma=1$ is another root of $P_2(\sigma;M)$ such that $H_1(1;M)>0$. By Theorem \ref{prop:sigpoly}, $\mathcal{M}$ is not injective and hence the proof is complete.
\end{proof}

\section{Hyperbolicity} \label{sec:hyp}
In this section we prove Theorems \ref{prop:schy2} \& \ref{thm:hyp2}.
To this end, it suffices to show that the characteristic polynomial
\begin{equation}
  c=c(u;M)= u^{2n+1} -\sum_{j=0}^{2n}a_j u^j
\end{equation}
of the coefficient matrix $A(M)$ in (\ref{eq:defA}) has distinct real roots.

\subsection{A proof of Theorem \ref{prop:schy2}} \label{sec:nhyp}

First of all, we follow \cite{Huang2020} and introduce an auxiliary polynomial associated with the characteristic polynomial:
\begin{equation} \label{eq:gdef}
  g=g(u;W) = \Delta_{2n+1}(u,\sigma)-\sum_{j=0}^{2n}a_j \Delta_j(u,\sigma).
\end{equation}
With $g$ thus defined, we claim that
\begin{equation} \label{eq:cg}
  c(u;M)=\sum_{k=0}^{2n+1}b_k\sigma^k \partial_u^k g(u;W).
\end{equation}
Indeed, we recall \ref{eq:Mstar} \& \ref{eq:defmstar} that
\[
 \sum_{i=1}^nw_iu_i^j=M_j^* = \sum_{k=0}^j b_k \sigma^k \frac{j!}{(j-k)!}\sum_{i=1}^nw_i\Delta_{j-k}(u_i,\sigma).
\]
In this identity, we set $w_1=1,w_2=\cdots =w_n=0$ and derive from (\ref{eq:dudelta}) that
\[
 u_1^j=\sum_{k=0}^j b_k \sigma^k \frac{j!}{(j-k)!}\Delta_{j-k}(u_1,\sigma)=\sum_{k=0}^j b_k \sigma^k\partial_u^k\Delta_j(u_1,\sigma).
\]
Thus the claim becomes clear.

Furthermore, $g(u;W)$ has the following elegant property for general kernel $\calk(\xi)$.

\begin{proposition} \label{prop:gform}
  \[
    g(u;W) = (u-u_1)^2 \cdots (u-u_n)^2 (u-\tilde{u}(W)).
  \]
\end{proposition}

\begin{proof}
  With the expressions for $M_j$ in (\ref{eq:MWmap}) and $\Delta_j(u,\sigma)$ in (\ref{eq:deltajex}) ($j=0,\dots,2n$), we calculate the $(j+1)$th-row of the Jacobian matrix $\frac{\partial \mathcal M}{\partial W} \in \mathbb R^{(2n+1)\times (2n+1)}$ as
  \[
    \left( \Delta_j(u_1,\sigma), w_1 \partial_{u}\Delta_j(u_1,\sigma),\dots,\Delta_j(u_n), w_n\partial_u\Delta_j(u_n,\sigma), \partial_{\sigma}M_j \right).
  \]
  Moreover, from (\ref{eq:mlast}) we compute $\frac{\partial \mathcal M_{2n+1}}{\partial W}$ as
  \[
    \left( \Delta_{2n+1}(u_1,\sigma),w_1\partial_u\Delta_{2n+1}(u_1,\sigma),\dots,\Delta_{2n+1}(u_n,\sigma),w_n\partial_u\Delta_{2n+1}(u_n,\sigma),\partial_{\sigma}M_{2n+1} \right).
  \]
  Then we use the relation
  \begin{equation} \label{eq:aj}
    (a_0,\dots,a_{2n})\frac{\partial \mathcal M}{\partial W}
    = \frac{\partial \mathcal M_{2n+1}}{\partial M} \frac{\partial \mathcal M}{\partial W}
    = \frac{\partial \mathcal M_{2n+1}}{\partial W}
  \end{equation}
  to obtain, for $i=1,\dots,n$,
  \begin{subequations}
  \begin{align}
    a_0 \Delta_0(u_i,\sigma) + \cdots + a_{2n} \Delta_{2n}(u_i,\sigma) &= \Delta_{2n+1}(u_i,\sigma), \label{eq:sub1} \\
    a_0 \partial_u\Delta_0(u_i,\sigma) + \cdots + a_{2n} \partial_u\Delta_{2n}(u_i,\sigma) &= \partial_u\Delta_{2n+1}(u_i,\sigma), \label{eq:sub2}
  \end{align}
  \end{subequations}
  implying $g(u_i)=\partial_u g(u_i)=0$ for $i=1,\dots,n$.
  This immediately leads to the expected expression.
\end{proof}

We also need the following elementary fact.
\begin{proposition} \label{prop:pfdf}
  If a polynomial $f(u)$ of degree $N$ has $N$ real roots  (counting multiplicity) and the maximum multiplicity of the roots is $m(f)$, then
  \[
    g_s(u)=f(u)+sf'(u)
  \]
  also has $N$ real roots (counting multiplicity) and $m(g_s)=\max\{m(f)-1,1\}$ if $s\ne 0$.
\end{proposition}

\begin{proof}
  According to the condition, we can write $f(u) = C(u-u_1)^{m_1}\cdots (u-u_r)^{m_r}$ with $\sum_{i=1}^r m_i = N$.
  Denote by $u_i' \in (u_i,u_{i+1})$ the root of $f'(u)$ for $i=1,\dots,r-1$. Set $u_0'=-\infty$ and $u_r'=\infty$.
  It suffices to show that, if $s\ne0$,
  \[
    g_s(u)= C \prod_{i=1}^r (u-u_i)^{m_i-1} \cdot \prod_{i=1}^r (u-v_i)
  \]
  with $v_i \in (u_{i-1}',u_i')$ and $v_i \ne u_i$ for $i=1,\dots,r$.

  For this purpose, we first notice that $u_i$ is a root of $g_s(u)$ with multiplicity $(m_i-1)$.
  Then we look into the interval $(u_{i-1}',u_i')$ which contains $u_i$ for $i=1,\dots,r$. Note $g_s(u_i')=f(u_i')$ for $i=0,...,r$ (this equality holds for $i=0,r$ where $f(u_i')= \pm\infty$).
  If $m_i$ is even, we have $f(u_{i-1}')f(u_i')>0$ while $g_s(u_i-\epsilon)g_s(u_i+\epsilon)<0$ because the multiplicity of $u_i$ is odd ($=m_i-1$) for $g_s(u)$. Therefore, there exists one root of $g_s(u)$, denoted by $v_i$, in either $(u_{i-1}',u_i-\epsilon)$ or $(u_i+\epsilon,u_i')$. It is hence distinct from $u_i$.
  Similarly, such a $v_i$ also exists for odd $m_i$.
  In this way, we get $r$ additional roots $v_i$ ($i=1,\dots,r$) of $g_s(u)$.
  These roots are all simple because the degree of $g_s(u)$ is $\sum_{i=1}^r (m_i-1) + r = N$.
  Hence, $g_s(u)$ can be factorized as above.
\end{proof}

Now we are in a position to prove Theorem \ref{prop:schy2}.
\begin{proof}[Proof of Theorem \ref{prop:schy2}]
  Referring to the condition of Theorem \ref{prop:schy2}, we can write the $b$-polynomial as $p(t)=(t+d_1)\cdots (t+d_{2n+1})$ with at least two of the $d_i$'s being nonzero. It is straightforward to verify that
  \[
    \sum_{j=0}^{2n+1} b_j t^j = t^{2n+1}p\left(\frac{1}{t}\right) = (1+d_1t)\cdots (1+d_{2n+1}t).
  \]
  Substituting $t$ with $\sigma \partial_u$, we obtain
  \[
    \sum_{j=0}^{2n+1}b_j \sigma^j \partial_u^j g(u;W) = (1+d_1\sigma \partial_u)\cdots (1+d_{2n+1}\sigma \partial_u) g(u;W),
  \]
  which is the characteristic polynomial $c(u;M)$ according to (\ref{eq:cg}).

  As shown in Proposition \ref{prop:gform}, $g(u;W)$ has $(2n+1)$ real roots (of $u$) and the maximum multiplicity $m(g)\le 3$ if $W\in \Omega$.
  Then by repeatedly using Proposition \ref{prop:pfdf} with $s=\sigma d_i$ ($i=2n+1,2n,\dots,1$), we see that the characteristic polynomial $c(u;M)$ has $(2n+1)$ real roots.
  Moreover, since at least two of the $d_i$'s are nonzero, the maximum multiplicity of each root is reduced to 1. Hence, all the roots are distinct.
\end{proof}

\subsection{A proof of Theorem \ref{thm:hyp2}} \label{sec:hyp2}

First of all, we recall (\ref{eq:cg}) for $n=2$ that the characteristic polynomial is
\[
  c(u;M)=\sum_{k=0}^5 b_k \sigma^k \partial_u^k g(u;W)
\]
and Proposition \ref{prop:gform} reads as
\[
  g(u;W)=(u-u_1)^2(u-u_2)^2(u-\tilde u(W)).
\]
Moreover, for even kernels we have $\frkm_1=\frkm_3=\frkm_5=0$. Then it follows from the definition (\ref{eq:bj}) that
\begin{equation} \label{eq:bj2}
  b_2 = -\frac{1}{2}, \quad b_4 = \frac{1}{4} - \frac{1}{24}\frkm_4, \quad b_1=b_3=b_5=0.
\end{equation}
Notice that $\tilde{u}=a_{2n}-2\sum_{i=1}^nu_i$. We can obtain $a_4$ by solving (\ref{eq:aj}) via Crammer's rule and thereby
\begin{equation} \label{eq:ut2}
  \tilde u(W) = \frac{w_1u_1+w_2u_2}{w_1+w_2} + \frac{4(\frkm_4-3)\sigma^2(w_1-w_2)(u_1-u_2)}{(w_1+w_2)[(u_1-u_2)^2+2(\frkm_4-3)\sigma^2]}.
\end{equation}
Clearly, $\tilde{u}(W)$ and thereby $g(u;W)$ are well-defined for $u_1=u_2$.

With these preparations, we turn to

\textit{The proof of Theorem \ref{thm:hyp2}. }
  By Corollary \ref{thm:inj} and Remark \ref{rem:Mtot}, the condition $\frkm_4 \ge 3$ ensures that the EQMOM is well defined. Assume that the two-node EQMOM is strictly hyperbolic on $\mathcal M(\Omega^{tot})$. Taking $W=(\frac{1}{2},0,\frac{1}{2},0,1)\in \Omega^{eq}$, we have $g(u;W)=u^5$ and $c(u;W)=u(u^4 -10u^2 + 120b_4)$.
  The strict hyperbolicity that $c(u;W)$ has 5 distinct real roots implies $b_4>0$ and hence $\frkm_4<6$ due to (\ref{eq:bj2}).

  Conversely, if $3\le \frkm_4<6$, we see from (\ref{eq:bj2}) that $b_2=-\frac{1}{2}<0<b_4\le \frac{1}{8} < \frac{5}{6}b_2^2 $.
  Thus, $c(u;W)=g(u;W)+b_2\sigma^2 \partial_u^2 g(u;W) + b_4\sigma^4 \partial_u^4 g(u;W)$ has 5 distinct real roots due to Proposition \ref{prop:cgp} to be proved below.

\begin{proposition}\label{prop:cgp}
  Set $g(u)=g(u;W)$. For any $s>0$, the polynomial
  \[
    h(u;s) = g(u) - c_1sg^{(2)}(u) + c_2s^2g^{(4)}(u)
  \]
  has 5 distinct roots if the constants satisfy $c_1 > 0$ and $0 < c_2 < \frac{5}{6} c_1^2$.
\end{proposition}

\begin{proof}
  Without loss of generality, we assume $u_1\leq u_2$ and consider four cases: (I) $u_1=u_2=\tilde u$, (II) $u_1<\tilde{u}< u_2$, (III) $u_1<u_2\leq\tilde{u}$ and (IV) $\tilde{u}\leq u_1<u_2$.

  \textbf{Case I}. In this case, we have $g(u) = (u-u_1)^5$ and
  \[
    h(u;s)=(u-u_1)^5-20c_1s(u-u_1)^3+120c_2s^2(u-u_1)=(u-u_1)(t^2-20c_1st+120c_2s^2)
  \]
  with $t=(u-u_1)^2$. Thanks to $c_1 > 0$ and $0 < c_2 < \frac{5}{6} c_1^2$, the quadratic function in $t$ has two distinct positive roots. Therefore $h(u;s)$ has 5 distinct real roots.

  For other cases, notice that $h(u;s)$ is a monic polynomial of order 5. Thus, it suffices to find $\hat u_1<\dots<\hat u_4$ such that
  \[
    h(\hat u_1;s)>0\quad\text{and}\quad h(\hat u_i;s)h(\hat u_{i+1};s)<0,\quad i=1,2,3.
  \]
  For this purpose, we set $P_u^+:=\{s>0|h(u;s)>0\}$ and $P_u^-:=\{s>0|h(u;s)<0\}$ for each $u$. Then our main tast is to show
  \begin{equation}\label{eq:Pu}
    \bigcup\limits_{u\in I} P_u^+ = (0,\infty)=\bigcup\limits_{u\in I'} P_u^-
  \end{equation}
  for some intervals $I,I'$. To do this, the key idea is to view $h(u;s)$ as a quadratic function of $s$ and check its symmetric axis and discriminant
  \[
    A(u) = \frac{c_1 g^{(2)}(u)}{2c_2 g^{(4)}(u)},\quad d(u)=g^{(2)}(u)^2-\beta g^{(4)}(u)g(u)
  \]
  with $\beta=4c_2/c_1^2\in(0,\frac{10}{3})$. If the leading coefficient $c_2g^{(4)}(u)$ is negative (or positive) and the discriminant $d(u)$ is positive in $I$ (resp. $I'$), then $P_u^+$ (resp. $P_u^-$) is an open interval centered at $s=A(u)$.

  In addition, we recall the form of $g(u)$ given in Proposition \ref{prop:gform} and denote by $u^{(j)}_1\leq\dots\leq u^{(j)}_{5-j}$ the $(5-j)$ roots of $g^{(j)}(u)$ (counting multiplicity). It is not difficult to see that $u^{(j)}_i \leq u^{(j+1)}_i \leq u^{(j)}_{i+1}$ for $j=1,2,3$ and $i=1,\dots,4-j$.
  The equalities occur only when $j=1$ and $\tilde{u}=u_1$ or $u_2$.

  \textbf{Case II}.
  In this case, we have $u_1=u_1^{(1)} < u^{(2)}_1 < u^{(1)}_2<\tilde u$ and $u^{(2)}_1<u^{(3)}_1<u^{(4)}_1$, resulting in $g(u^{(2)}_1)<0$ and $g^{(4)}(u^{(2)}_1)<0$. Thus, we have $h(u^{(2)}_1;s)<0$ for any $s>0$. A similar argument shows $h(u^{(2)}_3;s)>0$.
  As a consequence, it suffices to prove that there exist $\hat u_1\le u_1$ and $\hat u_4 \ge u_2$ such that $h(\hat u_1;s)>0$ and $h(\hat u_4;s)<0$; in other words, we need to verify (\ref{eq:Pu}) for $I=(-\infty,u_1]$ and $I'=[u_2,\infty)$.

  To this end, we notice $A(u) \to +\infty$ as $u \to -\infty$ and $P_{u_1}^+ = (0, 2A(u_1))$. Owing to the continuity of $A(u)$, it suffices to show that $d(u)>0$ for $u\leq u_1$.
  Notice that $d(u)$ is a polynomial of order 6. A straightforward calculation yields (the $u$-dependence is omitted for clarity)
  \begin{equation} \label{eq:dd1}
    \begin{aligned}
      d^{(1)} &=2g^{(2)}g^{(3)}-\beta(g^{(5)}g+g^{(4)}g^{(1)}), \\
      d^{(2)} &=2g^{(3)2}+(2-\beta)g^{(2)}g^{(4)}-2\beta g^{(1)}g^{(5)}, \\
      d^{(3)} &=(6-\beta)g^{(3)}g^{(4)}+(2-3\beta)g^{(2)}g^{(5)}, \\
      d^{(4)} &=(6-\beta)g^{(4)2}+(8-4\beta)g^{(3)}g^{(5)}, \\
      d^{(5)} &=(20-6\beta)g^{(4)}g^{(5)}, \\
      d^{(6)} &=(20-6\beta)g^{(5)2}.
    \end{aligned}
  \end{equation}
  Obviously, we have $d(u_1)>0$, $d^{(1)}(u_1)<0$, $d^{(5)}(u_1)<0$ and $d^{(6)}(u_1)>0$.
  Setting $a=u_2-u_1>0$ and $t=\frac{\tilde{u}-u_1}{u_2-u_1}\in(0,1)$, we obtain
  \begin{equation} \label{eq:dd2}
    \begin{aligned}
      d^{(2)}(u_1) &=24a^4[(16-2\beta)t^2+(20-4\beta)t+3], \\
      d^{(3)}(u_1) &=-48a^3[6(6-\beta)t^2+5(20-6\beta)t+6(6-\beta)], \\
      d^{(4)}(u_1) &=576a^2[(6-\beta)t^2+(44-14\beta)t+(34-9\beta)].
    \end{aligned}
  \end{equation}
  For $0<\beta<\frac{10}{3}$, the above $t$-parabolae are positive on $[0,\infty)$. Thus, we have $d^{(2)}(u_1)>0$, $d^{(3)}(u_1)<0$ and $d^{(4)}(u_1)>0$. The Taylor expansion of $d(u)$ at $u_1$ then verifies the positivity of $d(u)$ for $u\le u_1$.

  A similar argument can be used to show $\bigcup \limits_{u \ge u_2} P_u^- = (0,\infty)$.

  \textbf{Case III}. In this case we take $\hat u_2 = u^{(2)}_1$ and if $u_2\geq u^{(4)}_1$, we take $\hat u_3 = u_2$. It is then easy to see that $h(\hat u_2;s)<0<h(\hat u_3;s)$ for any $s>0$.

  For $u_2 < u^{(4)}_1$, we can show the existence of $\hat{u}_3\in(u_2, u^{(4)}_1)$ such that $h(\hat u_3,s)>0$ by verifying the first equality in (\ref{eq:Pu}) with $I=(u_2,u_1^{(4)})$.
  Clearly, for $u\in I$ we have $c_2 g^{(4)}(u)<0$, $P_{u_2}^+=(0,2A(u_2))$ and $A(u) \to \infty$ as $u\to u^{(4)}_1$. Then it suffice to show $d(u)>0$ for $u\in(u_2,u_1^{(4)})$.
  Note that $g^{(j)}(u)<0$ for $u\in (u_2,u^{(4)}_1)$ because $u_2=u^{(1)}_3 > u^{(2)}_2 > u^{(3)}_1$ and $u^{(4)}_1<u^{(3)}_2<u^{(1)}_4<\tilde u$.
  Then we see that $d(u_2)>0$ and $d^{(1)}(u_2)>0$.
  Moreover, if $\beta \le 2$, we see from (\ref{eq:dd1}) that $d^{(2)}(u)>0$ throughout $(u_2,u^{(4)}_1)$. Consequently, we have $d(u)>0$ for $u\in (u_2,u^{(4)}_1)$.
  On the other hand, if $\beta>2$, it follows from (\ref{eq:dd1}) that $d^{(3)}(u)>0$ for $u\in (u_2,u^{(4)}_1)$.
  As for $d^{(2)}(u_2)$, it has the same form as in (\ref{eq:dd2}) except that $u_1$ and $u_2$ are permuted. Clearly, we have $d^{(2)}(u_2)>0$. In conclusion, we have shown that $d(u)>0$ for $u\in(u_2,u^{(4)}_1)$ in both cases.

  To show the existence of $\hat u_4$, we verify the second equality in (\ref{eq:Pu}) with $I'=(\tilde{u},\infty)$.
  As above, it suffices to show $d(u)>0$ for $u>\tilde u$. From (\ref{eq:dd1}), we see that $d(\tilde u)>0$, $d^{(5)}(\tilde u)>0$ and $d^{(6)}(\tilde u)>0$.
  Furthermore, we set $a=\tilde{u}-u_1>0$ and $t=\frac{\tilde{u}-u_2}{\tilde{u}-u_1}\in[0,1)$. A straightforward calculation gives
  \[
    g^{(1)}(\tilde u) = a^4 p_1(t), \ g^{(2)}(\tilde u)=4a^3p_2(t), \ g^{(3)}(\tilde u) = 6a^2p_3(t), \ g^{(4)}(\tilde u)=48ap_4(t)
  \]
  with $p_1(t) = t^2$, $p_2(t) = t^2+t$, $p_3(t)=t^2+4t+1$ and $p_4(t)=t+1$. Owing to the following inequalities
  \[
    \begin{aligned}
      p_2(t) p_3(t) &= (t^2+t)(t^2+4t+1) \ge (t^2+t)\cdot6t = 6p_1(t)p_4(t), \\
      p_3(t)^2 &=((t+1)^2+2t)^2 \ge 8t(t+1)^2 = 8p_2(t)p_4(t), \\
      p_3(t)^2 &=(t^2+4t+1)^2 \ge 36t^2 = 36p_1(t), \\
      p_3(t)p_4(t) &=(t^2+4t+1)(t+1) \ge 6t(t+1) = 6p_2(t), \\
      3p_4(t)^2 &= 3(t+1)^2 \ge 4t + 2(t+1)^2 = 2p_3(t),
    \end{aligned}
  \]
  we derive, by a direct calculation, that
  \[
    \begin{aligned}
      d^{(1)}(\tilde u) &= 48a^5 \left(p_2(t)p_3(t) - \beta p_4(t) p_1(t) \right) \ge 48a^5 (6-\beta) p_1(t) p_4(t) >0, \\
      d^{(2)}(\tilde u) &= 72a^4 p_3(t)^2 - 192(\beta-2)a^4 p_2(t) p_4(t) -240\beta a^4 p_1(t) \\
      & \ge (256-192(\beta-2)) a^4 p_2(t) p_4(t) + 240(6-\beta) a^4 p_1(t) >0, \\
      d^{(3)}(\tilde u) &= 288(6-\beta)a^3p_3(t)p_4(t) - 480(3\beta-2) a^3 p_2(t) \ge 96(118-33\beta) a^3 p_2(t)>0, \\
      d^{(4)}(\tilde u) &=48^2(6-\beta)a^2 p_4(t)^2 - 720(4\beta-8)a^2 p_3(t) \ge 192(78-23\beta) a^2 p_3(t) > 0.
    \end{aligned}
  \]
  These ensure that $d(u)>0$ and thereby the existence of $\hat u_4$.

  A similar argument as in Case (II) can be used to prove the existence of $\hat u_1 \le u_1$ such that $h(\hat u_1;s)>0$ for any $s>0$.

  \textbf{Case IV}. This case can be converted to Case (III) with $-u_2<-u_1 \le -\tilde u$ by introducing $\tilde{g}(u)=-g(-u) = (u+u_1)^2(u+u_2)^2(u+\tilde{u})$. We see that $-h(-u;s)=\tilde{g}(u)- c_1s\tilde{g}^{(2)}(u) + c_2s^2\tilde{g}^{(4)}(u)$ has 5 distinct roots, so is $h(u;s)$. Hence, the proof is complete.
\end{proof}

\section{Structural stability} \label{sec:stab}

In this section we prove Theorem \ref{thm:stab}, namely, checking the structural stability condition (I)--(III) in Subsection \ref{sec:ssc}.

For Condition (I), we calculate the Jacobian of $S(M) = \frac{1}{\tau} \left(\rho\Delta^{eq}(U,\sqrt{\theta}) -M \right)$ as
\begin{equation} \label{eq:smm}
  S_M(M) =\frac{1}{\tau}
  \begin{bmatrix}
    0 &&&& \\
    & 0 &&& \\
    && 0 && \\
    s_1 & s_2 & s_3 & -1 & \\
    s_4 & s_5 & s_6 && -1
  \end{bmatrix}
\end{equation}
with
\[
\begin{aligned}
  s_1=U^3-3U\theta,\quad &s_2=3(\theta-U^2),\quad s_3 = 3U, \\
  s_4=3(U^4-2U^2\theta-\theta^2),\quad &s_5=-8U^3,\quad s_6=6(U^2+\theta).
\end{aligned}
\]
Take
\begin{equation} \label{eq:Pm}
  P(M) =
  \begin{bmatrix}
    1 &&&& \\
    & 1 &&& \\
    && 1 && \\
    -s_1 & -s_2 & -s_3 & 1 & \\
    -s_4 & -s_5 & -s_6 && 1
  \end{bmatrix},
\end{equation}
It is obvious that $P(M)S_M(M)=\tau^{-1}\text{diag}(0,0,0,-1,-1)P(M)$. This justifies Condition (I).  Note that the choice of $P$ is unique up to a block-diagonal matrix.

As to Condition (II), we know from Theorem \ref{thm:hyp2} that the two-node moment system (\ref{eq:eqmomsys}) with even kernels is strictly hyperbolic if $3\le \frkm_4 < 6$. Namely, the coefficient matrix $A(M)$ has 5 distinct real eigenvalues $\lambda_i=\lambda_i(M) (i=1,...,5)$. Corresponding to these eigenvalues, the left eigenvectors form the following matrix
\begin{equation} \label{eq:L}
  L=L(M)=
  \begin{bmatrix}
    \lambda_1^4 & \lambda_1^3 & \lambda_1^2 & \lambda_1 & 1 \\
    \lambda_2^4 & \lambda_2^3 & \lambda_2^2 & \lambda_2 & 1 \\
    \lambda_3^4 & \lambda_3^3 & \lambda_3^2 & \lambda_3 & 1 \\
    \lambda_4^4 & \lambda_4^3 & \lambda_4^2 & \lambda_4 & 1 \\
    \lambda_5^4 & \lambda_5^3 & \lambda_5^2 & \lambda_5 & 1
  \end{bmatrix}
  \begin{bmatrix}
    -1 &&&& \\
    a_4 & -1 &&& \\
    a_3 & a_4 & -1 && \\
    a_2 & a_3 & a_4 & -1 & \\
    a_1 & a_2 & a_3 & a_4 &-1
  \end{bmatrix},
\end{equation}
which can be easily verified. With this matrix, the symmetrizer $A_0=A_0(M)$ in Condition (II) must be chosen as $A_0=L^T \Lambda L$ with $\Lambda$ an arbitrary positive definite diagonal matrix to be determined \cite{Yong1999}.

The rest of this section is to choose the diagonal matrix $\Lambda=\Lambda(M)$ such that Condition (III) is satisfied for $M$ in the equilibrium manifold.
Since $P(M)S_M(M)=\tau^{-1}\text{diag}(\mathbf{0}_3,-I_2)P(M)$, it is equivalent to find $\Lambda$ such that the matrix $P^{-T}A_0P^{-1}=P^{-T}L^T \Lambda LP^{-1}$ is block diagonal with the same partition as $\text{diag}(\mathbf{0}_3,-I_2)$, meaning that the first three columns of $\sqrt\Lambda LP^{-1}$ are orthogonal to the last two columns.
Note that the existence of such a $\Lambda$ is independent of the choice of $P$.
Denote by $r_i\in\mathbb R^5$ the $i$th column of $\sqrt\Lambda L$.
Since
\[
  P^{-1} =
  \begin{bmatrix}
    1 &&&& \\
    & 1 &&& \\
    && 1 && \\
    s_1 & s_2 & s_3 & 1 & \\
    s_4 & s_5 & s_6 && 1
  \end{bmatrix},
\]
the orthogonality gives six equations
\begin{equation} \label{eq:relcond3}
    (r_i, r_j) + s_i (r_4, r_j) + s_{i+3} (r_5, r_j) = 0 \quad \text{for } i=1,2,3 \text{ and } j=4,5.
\end{equation}
Here $(\cdot,\cdot)$ represents the dot product of vectors.

To show that (\ref{eq:relcond3}) can be used to determine $\Lambda$, we write $\sqrt\Lambda=\text{diag}(x_i)_{i=1}^5$. Then it follows from (\ref{eq:L}) that
\begin{equation} \label{eq:ri}
   r_i = \left( x_1\sum_{j=i}^5 a_j \lambda_1^{j-i}, \ \dots, \ x_5\sum_{j=i}^5 a_j \lambda_5^{j-i} \right)^T
\end{equation}
with $a_5=-1$. With this, the dot products can be written as
\[
  (r_k,r_5) = -\sum_{i=1}^5 x_i^2 \sum_{j=k}^5 a_j \lambda_i^{j-k}.
\]
Moreover, we introduce
\[
   r_4'=r_4+a_4r_5=-(x_1\lambda_1,\dots,x_5\lambda_5)^T.
\]
It is clear that the equations in (\ref{eq:relcond3}) with $j=4$ can be replaced with
\[
  (r_i, r_4') + s_i (r_4, r_4') + s_{i+3} (r_5, r_4') = 0,\quad i=1,2,3,
\]
and
\begin{equation} \label{eq:r4pr}
\begin{split}
  (r_k,r_4') &=\sum_{i=1}^5 x_i^2 \left[ \lambda_i \left( \lambda_i^{5-k} - a_4 \lambda_i^{4-k} - \cdots - a_k \right) -a_{k-1}+a_{k-1} \right] \\
  &=(r_{k-1},r_5)+a_{k-1}\sum_{i=1}^5x_i^2.
\end{split}
\end{equation}
These indicate that (\ref{eq:relcond3}) is a system of six linear equations for the five unknowns $x_i^2$. Note that the coefficients of this system all depend on $M$.

Since Condition (III) is posed only for $M$ in the equilibrium, we only need to calculate the coefficients for $M$ in the equilibrium manifold for the moment system (\ref{eq:eqmomsys}):
\[
  \mathcal E=\{ M\in \mathcal{M}(\Omega^{tot}) | M_j= \rho \Delta_j^{eq}(U,\sqrt{\theta}) \text{ for }j=0,\dots,4 \}
\]
with $\Delta_j^{eq}(U,\sqrt{\theta})$ defined in (\ref{eq:unc}) and
\[
  \rho=M_0,\quad U=M_1/M_0\quad \text{and}\quad \theta=(M_0M_2-M_1^2)/M_0^2.
\]
About this $\mathcal{E}$, we have
\begin{proposition} \label{prop:eq2}
  For any $\rho,\theta>0$ and $U\in\mathbb{R}$, consider equations
  \[
    M_j:=\sum_{i=1}^2w_i\Delta_j(u_i,\sigma)=\rho \Delta_j^{eq}(U,\sqrt{\theta}),\quad j=0,1,...,4,
  \]
  for $W=(w_1,u_1,w_2,u_2,\sigma)\in\Omega^{tot}$. When $\frkm_4<3$, the equations have no solution; when $\frkm_4=3$, the solutions satisfy $u_1=u_2=U,\sigma=\sqrt{\theta}$ and $w_1+w_2=\rho$; and when $\frkm_4>3$, there is a unique solution given as
  \[
    w_{1,2}=\frac{\rho}{2}, \quad
    u_{1,2}=U \mp \nu\sigma, \quad
    \sigma = \left(\frac{\theta}{\nu^2+1}\right)^{\frac{1}{2}},\quad \nu = \left(\frac{\frkm_4-3}{2}\right)^{\frac{1}{4}}.
  \]
  In particular, the equilibrium manifold is nonempty if and only if $\frkm_4\geq3$.
\end{proposition}

\begin{proof}
  Notice that
  \begin{align*}
   \sum_{i=1}^2 w_i \Delta_j \left( \frac{u_i-U}{\sigma}, 1 \right) =
   &\int \xi^j \sum_{i=1}^2 w_i \calk \left( \xi-\frac{u_i-U}{\sigma} \right) d\xi \\
   =& \int \left( \frac{\eta-U}{\sigma}\right)^j \sum_{i=1}^2 \frac{w_i}{\sigma} \calk \left( \frac{\eta-u_i}{\sigma} \right) d\eta
   =\sum_{k=0}^j\binom{j}{k}(-U)^k\frac{M_{j-k}}{\sigma^j}
  \end{align*}
  and similarly,
  \[
    \rho\Delta_j^{eq} \left( 0,\frac{\sqrt{\theta}}{\sigma}\right)
    = \sum_{k=0}^j \binom{j}{k} (-U)^k \frac{\rho\Delta_{j-k}^{eq}(U,\sqrt{\theta})}{\sigma^j}.
  \]
  Then the given equations $M_j=\rho\Delta_j^{eq}(U,\sqrt{\theta})$ are equivalent to
  \begin{equation} \label{eq:Mjeqtp}
    \sum_{i=1}^2 \frac{w_i}{\rho} \Delta_j \left(\frac{u_i-U}{\sigma},1 \right)
    = \Delta_j^{eq} \left(0,\frac{\sqrt{\theta}}{\sigma}\right),\quad j=0,...,4.
  \end{equation}
  Denote $w_i'=\frac{w_i}{\rho}$, $u_i'=\frac{u_i-U}{\sigma}$ and $\theta'=\frac{\theta}{\sigma^2}$. Recall from (\ref{eq:deltajex}) that
  \[
    \Delta_j(u,\sigma) = \sum_{k=0}^j \binom{j}{k} \frkm_k \sigma^k u^{j-k}
  \]
  with $\frkm_3=0$ for even kernels. Then equations (\ref{eq:Mjeqtp}) can be rewritten as
  \begin{subequations}
  \begin{align}
      w_1' + w_2' &= 1, \label{eq:eqa} \\
      w_1'u_1' + w_2'u_2' = w_1'u_1'^3 + w_2'u_2'^3 &=0, \label{eq:eqb} \\
      w_1'u_1'^2 + w_2'u_2'^2 &= \theta' - 1, \label{eq:eqc} \\
      w_1'u_1'^4 + w_2'u_2'^4 &= 3\theta'^2-6(\theta'-1)-\frkm_4. \label{eq:eqd}
  \end{align}
  \end{subequations}
  Multiplying both sides of (\ref{eq:eqc}) with $(u_1'+u_2')$ and using (\ref{eq:eqb}), we obtain
  \[
    (\theta'-1)(u_1'+u_2')=0.
  \]
  This gives $u_1'=-u_2':=\nu\neq0$ if $\theta'\neq 1$.
  Then we see from (\ref{eq:eqa}) and (\ref{eq:eqb}) that $w_1'=w_2'=\frac{1}{2}$. Thus, (\ref{eq:eqd}) becomes $2\nu^4 = \frkm_4 - 3$ with $\theta'=1+\nu^2$.
  When $\theta'=1$, it follows from (\ref{eq:eqc}) that $u_1'=u_2'=0$ due to $w_i'>0$ and $\frkm_4=3$ from (\ref{eq:eqd}).
  Hence, the given equations have a solution if and only if $\frkm_4\ge 3$.
\end{proof}

At equilibrium, it is seen from Proposition \ref{prop:eq2} and (\ref{eq:ut2}) that $\tilde{u}=U$ and
\[
  g(u;W)=(u-U-\sigma\nu)^2(u-U+\sigma\nu)^2(u-U).
\]
We then rewrite $g(u;W)$ and
\[
  c(u;W)=g(u;W)+b_2\sigma^2 g^{(2)}(u;W) + b_4\sigma^4 g^{(4)}(u;W),
\]
with $b_j$ in (\ref{eq:bj2}) as $g(u;U,\sigma)$ and $c(u;U,\sigma)$ to manifest the dependence on $U$ and $\sigma$.
Clearly, we have $g(u;U,\sigma)=\sigma^5 g\left(\frac{u-U}{\sigma};0,1\right)$ and $c(u;U,\sigma)=\sigma^5 c\left(\frac{u-U}{\sigma};0,1\right)$.

A direct calculation yields $c(u;0,1)=u^5-B_1u^3+B_2u$ with
\begin{equation} \label{eq:B12}
  B_1 = 2\nu^2+10 \quad \text{and } B_2= -9\nu^4 + 6\nu^2 + 15.
\end{equation}
Recalling $c(u;W)=u^5-a_4u^4-\cdots-a_0$, we derive
\begin{equation} \label{eq:ajeq}
\begin{split}
  a_0 &= U^5 -B_1 U^3\sigma^2 + B_2 U \sigma^4, \
  a_1 =-5U^4 + 3B_1U^2\sigma^2 - B_2\sigma^4, \\
  a_2 &= 10U^3-3B_1U\sigma^2, \quad
  a_3 = -10U^2 + B_1\sigma^2, \quad
  a_4 = 5U.
\end{split}
\end{equation}
This determines the matrix $L=L(M)$ in (\ref{eq:L}) on the equilibrium manifold.

Denote by $\pm \mu_1$, $\pm \mu_2$ and 0 the five distinct roots of the polynomial $c(u;0,1)$. Then the roots $\lambda_i$ of $c(u;U,\sigma)$ can be written as
\begin{equation} \label{eq:croots}
    \lambda_{1,2} = U \pm \mu_1 \sigma, \quad
    \lambda_{3,4} = U \pm \mu_2 \sigma, \quad
    \lambda_5 = U.
\end{equation}
By introducing $V=U/\sigma$, it is seen that $r_i$ is a homogeneous polynomial of $\sigma$ (of degree $5-i$), so the relations in (\ref{eq:relcond3}) are all homogeneous with $\sigma$.
Thus, we may set $\sigma=1$ for the following calculations.

Set
\begin{displaymath}
   Y_0=\sum_{i=1}^5 x_i^2,\quad
   Y_j=\left( x_1^2+(-1)^jx_2^2 \right)\mu_1^j + \left( x_3^2+(-1)^jx_4^2 \right)\mu_2^j,\quad\text{for } j=1,\dots, 4.
\end{displaymath}
Having (\ref{eq:ajeq}) \& (\ref{eq:croots}), the dot products $(r_i,r_5)$ can be calculated as
\[
  \begin{aligned}
    (r_1,r_5) &= (U^4-B_1U^2+B_2)Y_0 - (U^3-B_1U) Y_1 + (U^2-B_1)Y_2 - U Y_3 + Y_4, \\
    (r_2,r_5) &= (-4U^3+2B_1U)Y_0 + (3U^2-B_1)Y_1 - 2UY_2 + Y_3, \\
    (r_3,r_5) &= (6U^2-B_1)Y_0 - 3UY_1 + Y_2, \\
    (r_4,r_5) &= -4UY_0 + Y_1, \\
    (r_5,r_5) &= Y_0.
  \end{aligned}
\]
Furthermore, we use (\ref{eq:r4pr}) to get
\[
  \begin{aligned}
    (r_1,r_4') &= a_0 Y_0 = (U^5 -B_1 U^3 + B_2 U)Y_0, \\
    (r_2,r_4') &= (2B_1U^2-4U^4)Y_0 - (U^3-B_1U) Y_1 + (U^2-B_1)Y_2 - U Y_3 + Y_4, \\
    (r_3,r_4') &= (6U^3-B_1U)Y_0 + (3U^2-B_1)Y_1 - 2UY_2 + Y_3, \\
    (r_4,r_4') &= -4U^2Y_0 - 3UY_1 + Y_2, \\
    (r_5,r_4') &= UY_0 + Y_1.
  \end{aligned}
\]
With these and those in (\ref{eq:smm}), the equations in (\ref{eq:relcond3}) can be written as
\begin{subequations} \label{eq:feqn}
  \begin{align}
    0 =& \left[ (6\theta-B_1)U^2 + B_2-3\theta^2 \right] Y_0 + (B_1-3\theta)U Y_1 \label{eq:fa} \\ &+ (U^2-B_1)Y_2 - UY_3 + Y_4, \nonumber \\
    0 =& (2B_1-12\theta)UY_0 + (3\theta-B_1)Y_1 - 2UY_2 + Y_3, \label{eq:fb} \\
    0 =& (6\theta-B_1) Y_0 + Y_2, \label{eq:fc} \\
    0 =& \left[ (6\theta-B_1)U^3 + (B_2-3\theta^2)U \right] Y_0 + (3U^2\theta-3\theta^2)Y_1 + (U^3-3U\theta) Y_2, \label{eq:fd} \\
    0 =& (2B_1-12\theta)U^2Y_0 + (B_1-9\theta)UY_1 - (B_1-3\theta+2U^2)Y_2 -UY_3+Y_4, \label{eq:fe} \\
    0 =& (6\theta-B_1)UY_0 + (6\theta-B_1)Y_1 + UY_2 + Y_3. \label{eq:ff}
  \end{align}
\end{subequations}

We shall show that there exists $x_i^2$ solving these equations if and only if $\frkm_4<5$. To do this, we use (\ref{eq:fc}) and deduce from (\ref{eq:fb}) \& (\ref{eq:ff}) that $Y_3=(B_1-3\theta)Y_1=(B_1-6\theta)Y_1$ and hence $Y_1=Y_3=0$.
By the definitions of $Y_1$ and $Y_3$, it follows that $x_1^2=x_2^2$ and $x_3^2=x_4^2$ due to $\mu_1\ne\mu_2$.
On the other hand, notice that $\theta=\nu^2+1$ due to Proposition \ref{prop:eq2} with $\sigma=1$.
We deduce from (\ref{eq:B12}) that
\begin{equation} \label{eq:coef}
    B_2-3\theta^2 = 3\theta(B_1-6\theta) = -12(\nu^4-1) = -6(\frkm_4-5).
\end{equation}
Thus, it is not difficult to see that all the equations in (\ref{eq:feqn}) are linear combinations of (\ref{eq:fa}) and (\ref{eq:fc}).

Furthermore, since $\mu_1$ and $\mu_2$ are the nonzero roots of $c(u;0,1)=u^5-B_1u^3+B_2u$, we have $\mu_1^2+\mu_2^2=B_1>0$ and $\mu_1^2\mu_2^2=B_2>0$.
Thus, by the definitions of $Y_2$ and $Y_4$, we have $B_1Y_2 = Y_4 + B_2(x_1^2+\cdots+x_4^2)$.
Consequently, (\ref{eq:fa}) and (\ref{eq:fc}) are equivalent to (using $x_1^2=x_2^2$ and $x_3^2=x_4^2$)
\begin{subequations}
   \begin{align}
     (B_2-3\theta^2)Y_0&= B_2 \cdot 2(x_1^2+x_3^2),\label{eq:ra}\\
     Y_2&= (B_1-6\theta)Y_0.\label{eq:rb}
   \end{align}
\end{subequations}
Therefore, if $\frkm_4\geq5$, from (\ref{eq:coef}) and (\ref{eq:rb}) we see that $Y_0$ and $Y_2$ cannot be both positive. This together with the definitions of $Y_0$ and $Y_2$ indicates the nonexistence of the $x_i^2$'s and thereby the diagonal positive matrix $\Lambda$.

Finally, we show the existence of the $x_i^2$'s if $\frkm_4<5$. From (\ref{eq:coef}) we see that $\frkm_4<5$ if and only if $0\leq\nu<1$. Substituting (\ref{eq:coef}) into (\ref{eq:ra}) and using $Y_0=\sum_{i=1}^5x_i^2$, we easily obtain
\[
   x_1^2+x_3^2=2\frac{1-\nu^2}{\theta}x_5^2.
\]
Substituting this into (\ref{eq:rb}) and putting them into matrix form, we arrive at
\[
    \begin{bmatrix}
      1 & 1 \\ \mu_1^2 & \mu_2^2
    \end{bmatrix}
    \begin{bmatrix}
      x_1^2 \\ x_3^2
    \end{bmatrix}
    = \frac{2(1-\nu^2)}{\theta}x_5^2
    \begin{bmatrix}
      1 \\ 5-3\nu^2
    \end{bmatrix}.
\]
Assume $\mu_1<\mu_2$. It is not difficult to verify that the two components $x_1^2$ and $x_3^2$ of the solution to this system are positive if and only if $\mu_1^2 < 5-3\nu^2 < \mu_2^2$.
Notice that
\[
   \mu_2^2 > \frac{\mu_1^2+\mu_2^2}{2} = \frac{B_1}{2} = \nu^2+5 \geq 5-3\nu^2\quad\text{and}
\]
\[
   \mu_1^2 = \frac{B_1}{2}-\sqrt{\frac{1}{4}B_1^2-B_2} = 5+\nu^2-\sqrt{10\nu^4+4\nu^2+10}< 5-3\nu^2
\]
for $0\le \nu<1$. Hence, the existence of positive $x_1^2$ and $x_3^2$ is demonstrated and the proof is complete.

\section{Specific kernel functions} \label{sec:exp}
In this section we present a number of specific kernel functions which satisfy the conditions required by our previous theoretical results, and then examine their performance with numerical tests.

\subsection{Examples of kernel functions}
Here are examples of the kernel functions, which may not be normalized. To check the conditions in our main results, we notice that, for even normalized kernels,
\[
  \frkm_4=\int\sqrt{\tilde{\frkm}_2}\calk(\sqrt{\tilde{\frkm}_2}\xi)\xi^4d\xi=\tilde{\frkm}_2^{-2}\int \calk(\eta)\eta^4d\eta=\tilde{\frkm}_4/\tilde{\frkm}_2^2,
\]
where $\tilde{\frkm}_4,\tilde{\frkm}_2$ are the $4^{th}$ and $2^{nd}$ moments of the unnormalized kernel, respectively.

\begin{example}[Gaussian distribution] \label{ex:gauss}
  Our first example is the most widely-used Gaussian distribution
  \[
    \calk(\xi)=\frac{1}{\sqrt{2\pi}}\exp \left( -\frac{\xi^2}{2} \right)
  \]
  in the EQMOM approach for the BGK equation \cite{Chalons2017}.
  It is even and normalized with $\frkm_4=3$. According to Theorem \ref{thm:stab}, the two-node Gaussian-EQMOM satisfies the structural stability condition.
  Indeed, for this kernel, it has been even shown in \cite{Huang2020} that the structural stability condition is fulfilled for the $n$-node EQMOM.
\end{example}

\begin{example} [Piecewise polynomials] \label{ex:evenp}
  Our next example is the piecewise polynomials
  \[
    \calk_j(\xi)=\frac{j+1}{2}(1-|\xi|)^j \mathbf 1_{|\xi|\le 1},\quad j=0,1,\dots.
  \]
  These are even functions and $\calk_1$ was used in \cite{CR14}. A straightforward calculation of the unnormalized $k$th-moment $\tilde{\frkm}_{j,k}$ gives
  \[
    \frac{\tilde{\frkm}_{j,4}}{\tilde{\frkm}_{j,2}^2}
   = \frac{6(j+2)(j+3)}{(j+4)(j+5)}
 \]
 which is in $[3,6)$ for $j\ge 3$ and in $[3,5)$ for $3\le j\le 18$.
 According to Theorems \ref{thm:hyp2} \& \ref{thm:stab}, the corresponding two-node EQMOM moment system is well-defined and strictly hyperbolic if $j\ge 3$, and satisfies the structural stability condition if $3\le j \le 18$.

 On the other hand, even polynomials
 \[
  \tilde \calk_j(\xi) = \frac{j+1}{2j}(1-|\xi|^j)\mathbf 1_{|\xi|\le 1}, \quad j=1,2,\dots
 \]
 cannot be used as EQMOM kernels since $\tilde{\frkm}_{j,4}/\tilde{\frkm}_{j,2}^2=\frac{9(j+3)^2}{5(j+1)(j+5)} < 3$, violating the conditions of Corollary \ref{thm:inj}.
\end{example}

\begin{example}[Kappa distribution] \label{ex:kappa}
  For kernels to be the kappa distribution \cite{PL2010}
  \[
    \calk(\xi) = \frac{\Gamma(\kappa+1)}{\kappa \sqrt{\pi\left( \kappa-\frac{3}{2} \right)}\Gamma\left(\kappa-\frac{1}{2}\right)} \left( \frac{\xi^2}{\kappa-\frac{3}{2}}+1 \right)^{-\kappa}
  \]
  with $\kappa>3$, where $\Gamma(z)=\int_0^{\infty} t^{z-1}e^{-t}dt$ is the gamma function,
  we claim that the two-node EQMOM map $\mathcal M$ in (\ref{eq:MWmap}) is injective. The resultant moment system is strictly hyperbolic if $\kappa>\frac{7}{2}$ and satisfies the structural stability condition if $\kappa>4$.

  To see this, we only need to consider the rescaled kernel $\calk(\xi)= I_{0,\kappa}^{-1} (\xi^2+1)^{-\kappa}$ with $I_{j,\kappa} = \int_{\mathbb R} \xi^{2j} (\xi^2+1)^{-\kappa} d\xi$.
  Obviously we have
  \[
    I_{j,\kappa} = \int_{\mathbb R} (\xi^{2j} - \xi^{2j-2}+\xi^{2j-2}) (\xi^2+1)^{-\kappa} d\xi = I_{j-1,\kappa-1}+I_{j-1,\kappa}.
  \]
  Moreover, integrating by parts gives
  \[
    I_{j,\kappa} = \frac{-1}{2(\kappa-1)} \int_{\mathbb R} \xi^{2j-1} d (\xi^2+1)^{1-\kappa} = \frac{2j-1}{2(\kappa-1)}I_{j-1,\kappa-1}
  \]
  for $\kappa>j+1/2$. Thus we obtain $I_{j,\kappa} = \frac{2j-1}{2j-2\kappa+1}I_{j-1,\kappa}$ and
  \[
    \frac{\tilde{\frkm}_4}{\tilde{\frkm}_2^2} = \frac{I_{2,\kappa}/I_{0,\kappa}}{\left( I_{1,\kappa}/I_{0,\kappa} \right)^2}
    = \frac{I_{2,\kappa}/I_{1,\kappa}}{I_{1,\kappa}/I_{0,\kappa}}
    =\frac{3(2\kappa-3)}{2\kappa-5} > 3.
  \]
  Therefore, by Corollary \ref{thm:inj} the two-node EQMOM is well defined. Moreover, we have $\tilde{\frkm}_4/\tilde{\frkm}_2^2<6$ and $\tilde{\frkm}_4/\tilde{\frkm}_2^2<5$ if $\kappa>\frac{7}{2}$ and $\kappa>4$, respectively.
  Consequently, the conditions of Theorems \ref{thm:hyp2} \& \ref{thm:stab} are satisfied if $\kappa>\frac{7}{2}$ or $\kappa>4$.
\end{example}

\begin{example}
  Let
  \[
    \calk_j(\xi) = \frac{1}{2\Gamma(j+1)}|\xi|^j e^{-|\xi|},\quad j=1,2,\dots.
  \]
  A straightforward calculation gives $\tilde{\frkm}_{4,j}/\tilde{\frkm}_{2,j}^2=\frac{(j+3)(j+4)}{(j+1)(j+2)}$, which is greater than 3 if and only if $j=1$. For $\calk_1(\xi)$, the ratio is $\frac{10}{3}<5$. According to Theorem \ref{thm:stab}, the resultant two-node EQMOM moment system satisfies the structural stability condition.
\end{example}

\begin{example} \label{ex:nonhyp}
  Let
  \[
    \calk(\xi)=\frac{1}{3}(1+\cos\xi)e^{-|\xi|}.
  \]
  A direct calculation leads to $\tilde{\frkm}_4/\tilde{\frkm}_2^2=14$. According to Corollary \ref{thm:inj} and Theorem \ref{thm:stab}, the two-node EQMOM is well-defined, but the resultant moment system is not hyperbolic.
\end{example}

\begin{example}[Uneven kernels] \label{ex:uneven}
  Let us give an example of uneven kernels which allows a well-defined and strictly hyperbolic two-node EQMOM moment system.
  This kernel function reads as
  \[
    \calk(\xi) = \frac{1}{c} \left( \frac{1-\xi}{\xi^{\alpha}} \mathbf 1_{0<\xi\le 1} + \frac{1+\xi}{|\xi|^{\beta}} \mathbf 1_{-1\le\xi<0} \right)
  \]
  with $\alpha=0.6060$, $\beta=0.5340$ and $c=\frac{1}{(1-\alpha)(2-\alpha)} + \frac{1}{(1-\beta)(2-\beta)}=3.2845$.
  The normalized moments $\frkm_j$ of $\calk(\xi)$ are $\frkm_3=-0.04006$ and $\frkm_4=4.7455$.
  Thus, we have $\frkm_4>3+\frac{9}{8}\frkm_3^2$ and the corresponding two-node EQMOM is well-defined.
  Its strict hyperbolicity is ensured by Theorem \ref{prop:schy2} as the polynomial
  \[
    p(t)=t^5-0.5t^3+6.6766\times 10^{-3}t^2+0.052271t-2.7789\times 10^{-3}
  \]
  can be shown to have 5 distinct real roots (this can also be verified numerically).
\end{example}

\begin{example}
  Our last kernel function is
  \[
    \mathcal{K}(\xi)=\xi e^{-\xi}\mathbf{1}_{\xi>0}
  \]
  and its unnormalized moments are obviously $\tilde{\frkm}_j=(j+1)!$.
  It is not difficult to verify the conditions of Corollary \ref{thm:inj} and, thereby, the two-node EQMOM is well-defined.
  Moreover, a direct calculation via (\ref{eq:bj}) gives $b_1=-2$, $b_2=1$ and $b_j=0$ for $j\geq3$. Thus, the polynomial in Theorem \ref{prop:schy2} for the two-node case is $p(t)=t^3(t-1)^2$ and has two nonzero roots. By Theorem \ref{prop:schy2}, the two-node EQMOM moment system is strictly hyperbolic.
\end{example}

\subsection{Numerical validation} \label{subsec:num}
In this subsetion, we use a Riemann problem of the Euler equations to show that the kernel functions given in the previous subsection produce satisfactory results if they satisfy the conditions required by our theory; otherwise they lead to spurious results.
The initial data of the Riemann problem are
\[
  \rho(0,x) = \left\{
  \begin{aligned}
    3.093, \quad &x<0, \\
    1, \quad &x>0,
  \end{aligned}
  \right. \quad U(0,x)=0, \quad \theta(0,x)=1,
\]
for the Euler equation. These data are used to determine the equilibrium distribution $f^{eq}(\rho,U,\theta;\xi)$ in the kinetic equation (\ref{eq:1D-BGK}) and thereby initial moments.

The computational domain $-1 \le x \le 1$ is discretized into 1000 uniform cells. The Neumann boundary condition $\partial_x f=0$ is applied on the endpoints $x=\pm 1$.
The spatial fluxes are treated as in \cite{Chalons2017} and the time step is chosen so that the CFL number is less than 0.5.
Two limiting cases of $\tau$ in (\ref{eq:eqmomsys}) are considered here. The continuum limit has infinitely fast collisions with $\tau=0$, while the free-molecular limit assumes no collision ($\tau=\infty$). The analytical solutions for the two cases can be found in \cite{Toro2009} and \cite{Chalons2017}, respectively.
We test six different kernels for the both cases.

\begin{figure}[htbp]
  \centering
  \includegraphics[height=0.55\textwidth]{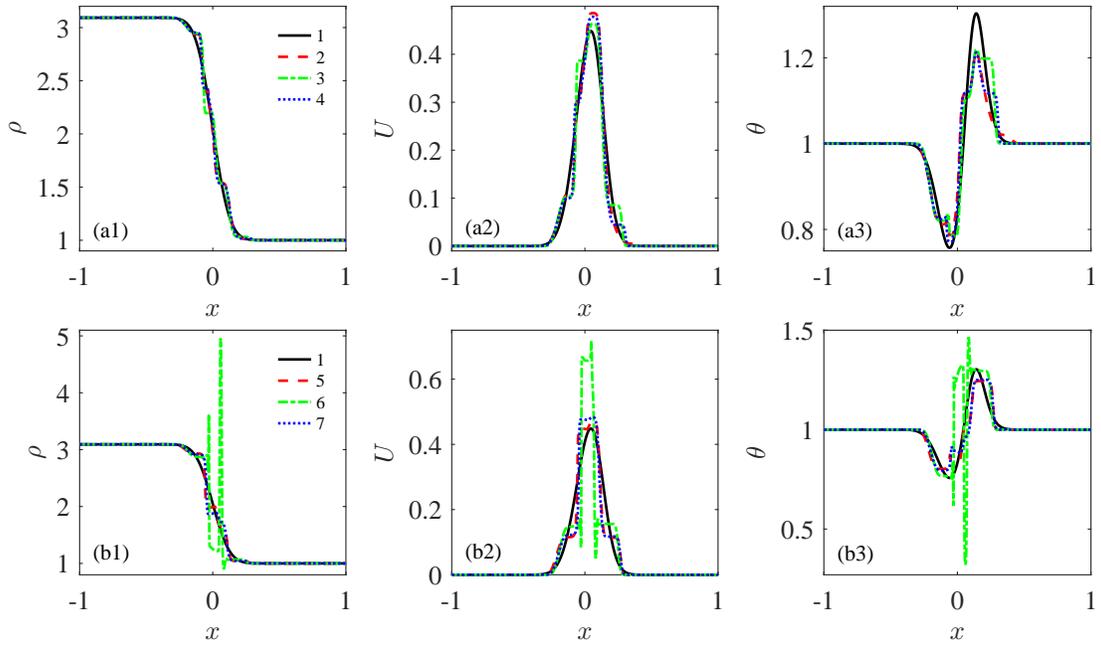}
  \caption{1-D Riemann problem with no collision: Profiles of density $\rho$, velocity $U$ and temperature $\theta$ at $t=0.1$. Kernels:
  1. Analytical solution; 2. Gaussian; 3. Kappa distribution with $\kappa=6$; 4. $K_3(\xi)$ in Example \ref{ex:evenp}; 5. $\calk(\xi)$ in Example \ref{ex:uneven}; 6. $\calk(\xi)$ in Example \ref{ex:nonhyp}; 7. $K_{25}(\xi)$ in Example \ref{ex:evenp}.}
  \label{fig:ncol}
\end{figure}

\begin{figure}[htbp]
  \centering
  \includegraphics[height=0.55\textwidth]{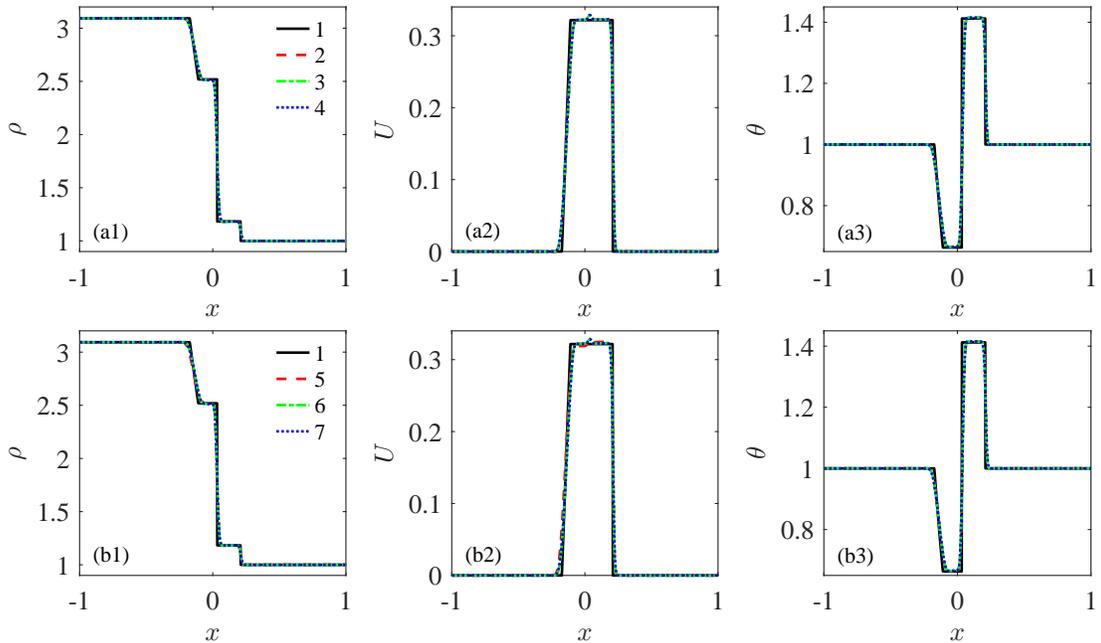}
  \caption{1-D Riemann problem in the continuum limit: Profiles of density $\rho$, velocity $U$ and temperature $\theta$ at $t=0.1$. Kernels:
  1. Analytical solution; 2. Gaussian; 3. Kappa distribution with $\kappa=6$; 4. $K_3(\xi)$ in Example \ref{ex:evenp}; 5. $\calk(\xi)$ in Example \ref{ex:uneven}; 6. $\calk(\xi)$ in Example \ref{ex:nonhyp}; 7. $K_{25}(\xi)$ in Example \ref{ex:evenp}.}
  \label{fig:col}
\end{figure}

Fig. \ref{fig:ncol} shows the spatial profiles of the macroscopic quantities $(\rho, u, \theta)$ at $t = 0.1$ for the free-molecular case.
Both the simulated results and analytical solutions are plotted.
It is seen that the kernel greatly affects the simulation results of this highly non-equilibrium flow. The Gaussian kernel, the kappa distribution in Example \ref{ex:kappa} with $\kappa=6$ and the even polynomial $\calk_3(\xi)$ in Example \ref{ex:evenp} all satisfy the structural stability condition, and exhibit similar precision in the simulation.
The results with the uneven kernel in Example \ref{ex:uneven} can still roughly reproduce the profiles but shows larger errors than the former kernels. Notice that the corresponding moment system is strictly hyperbolic, while whether the structural stability condition (III) is respected remains unclear.

On the other hand, Figs. \ref{fig:ncol} \textcolor{blue}{(b1)-(b3)} include the results from two improper kernels. The kernel in Example \ref{ex:nonhyp} yields a non-hyperbolic system, leading to huge unphysical peaks (termed `$\delta$-shocks') in the region where flow quantities change drastically. This is a common phenomenon for non-hyperbolic moment systems \cite{Fox2008}.
By contrast, the kernel $\calk_{25}(\xi)$ in Example \ref{ex:evenp} is hyperbolic but violates the stability condition (III). The errors are larger than those from the kernel $\calk_3(\xi)$.

Fig. \ref{fig:col} gives the numerical results of the continuum case. The flow can be reasonably predicted, including a right-moving shock wave, a left-moving rarefaction wave and a discontinuity at $x=0$. The result indicates that the continuum flow regime may be less sensitive to the choice of the kernels.

\section{Conclusions}
\label{sec:conclusions}

This paper is concerned with a class of moment closure systems derived with an extended quadrature method of moments (EQMOM) for the one-dimensional BGK equation.
The class is characterized with a kernel function and the unknown distribution is approximated with Ansatz Eq.(\ref{eq:feqmom}).
We investigate the realizability of the extended method of moments (see Theorem \ref{prop:sigpoly} \& Corollary \ref{thm:inj}).
A sufficient condition (Theorem \ref{prop:schy2}) on the kernel is identified for the EQMOM-derived moment systems to be strictly hyperbolic.

Furthermore, sufficient and necessary conditions are established for the two-node systems to be well-defined and strictly hyperbolic, and to preserve the dissipation property of the kinetic equation.
For normalized kernels, the condition is $\frkm_4 \ge 3+ \frac{9}{8}\frkm_3^2$ for the well-definedness, where $\frkm_3$ and $\frkm_4$ are the 3rd and 4th moments of the kernel function.
When the kernel is even, the conditions are $3\le \frkm_4<6$ for hyperbolicity and $3\le \frkm_4<5$ for the dissipativeness corresponding to the $H$-theorem of the kinetic equation.

In addition, we present a number of examples of the kernel functions and examine their performance numerically. Precisely, we use a Riemann problem of the Euler equations to show that the kernel functions produce satisfactory results if they satisfy the conditions required by our theoretical results; otherwise they lead to spurious results.

\section*{Acknowledgments}
\label{sec:acknowledgments}

This work is supported by the National Key Research and Development Program of China (Grant no. 2021YFA0719200) and the National Natural Science Foundation of China (Grant no. 12071246).
The authors are grateful to Prof. Shuiqing Li and Mr. Yihong Chen at Tsinghua University for insightful discussions.

\end{document}